\documentclass{gtmon_a}
\pdfoutput=1

\proceedingstitle{The Zieschang Gedenkschrift}
\conferencestart{5 September 2007}
\conferenceend{8 September 2007}
\conferencename{Conference in honour of Heiner Zieschang}
\conferencelocation{Toulouse, France}

\editor{Michel Boileau}
\givenname{Michel}
\surname{Boileau}

\editor{Martin Scharlemann}
\givenname{Martin}
\surname{Scharlemann}

\editor{Richard Weidmann}
\givenname{Richard}
\surname{Weidmann}

\title{Finite  groups acting on $3$--manifolds and cyclic branched coverings of knots}

\author{Mattia Mecchia}
\givenname{Mattia}
\surname{Mecchia}
\address{Universit\`a degli Studi di Trieste\\\newline
Dipartimento di Matematica e Informatica\\
34100 Trieste\\Italy}
\email{mecchia@dmi.units.it}
\urladdr{}

\dedicatory{To the memory of Heiner Zieschang}

\volumenumber{14}
\issuenumber{}
\publicationyear{2008}
\papernumber{018}
\startpage{393}
\endpage{416}

\doi{}
\MR{}
\Zbl{}

\arxivreference{} 


\received{14 July 2007}
\revised{2 March 2007}
\accepted{2 March 2007}
\published{29 April 2008}
\publishedonline{29 April 2008}
\proposed{}
\seconded{}
\corresponding{}
\version{}

\AtBeginDocument{\let\bar\wbar\let\tilde\wtilde\let\hat\what}
\let\Bbb\mathbb
\newtheorem{thm}{Theorem}
\newtheorem*{thmn}{Theorem}
\newtheorem{prop}{Proposition}

\def\O{{\cal O}}
\def\Z{\Bbb Z}

\def\A{\Bbb A}

\begin{document}

\begin{asciiabstract}
We are interested in finite groups acting orientation-preservingly on
3-manifolds (arbitrary actions, ie not necessarily free
actions). In particular we consider finite groups which contain an
involution with nonempty connected fixed point set.  This condition is
satisfied by the isometry group of any hyperbolic cyclic branched
covering of a strongly invertible knot as well as by the isometry
group of any hyperbolic 2-fold branched covering of a knot in the
3-sphere.  In the paper we give a characterization of nonsolvable
groups of this type. Then we consider some possible applications to
the study of cyclic branched coverings of knots and of hyperelliptic
diffeomorphisms of 3-manifolds. In particular we analyze the basic
case of two distinct knots with the same cyclic branched covering.
\end{asciiabstract}

\begin{htmlabstract}
We are interested in finite groups acting orientation-preservingly on
3&ndash;manifolds (arbitrary actions, ie not necessarily free
actions). In particular we consider finite groups which contain an
involution with nonempty connected fixed point set.  This condition is
satisfied by the isometry group of any hyperbolic cyclic branched
covering of a strongly invertible knot as well as by the isometry
group of any hyperbolic 2&ndash;fold branched covering of a knot in
S<sup>3</sup>.  In the paper we give a characterization of nonsolvable
groups of this type. Then we consider some possible applications to
the study of cyclic branched coverings of knots and of hyperelliptic
diffeomorphisms of 3&ndash;manifolds. In particular we analyze the
basic case of two distinct knots with the same cyclic branched
covering.
\end{htmlabstract}

\begin{abstract} 
We are interested in finite groups acting orientation-preservingly on $3$--manifolds (arbitrary actions, ie not necessarily free actions). In particular we  consider finite groups  which contain an involution with nonempty connected fixed point set.  This condition is satisfied by the isometry group of any hyperbolic cyclic branched covering of a strongly invertible knot  as well as by the isometry group of any hyperbolic $2$--fold branched covering of a knot  in $S^3$.   In the paper we give a characterization of nonsolvable groups of this type. Then we consider some possible applications to the study of cyclic branched coverings of knots and of hyperelliptic diffeomorphisms of $3$--manifolds. In particular we analyze the basic case of two distinct knots with the same cyclic branched covering.
\end{abstract}

\maketitle

\section{Introduction}

The following problem has been   diffusely studied in the literature:  which finite groups  admit an action on a  homology $3$--sphere.  The choice of the coefficients of the homology changes completely the situation.

If a finite  group $G$ acts freely on an integer homology $3$--sphere (and in particular on the standard $3$--sphere $S^3$), the group $G$ has periodic cohomology of period four. Milnor \cite{Mn} gave a list of groups  which are candidates for free actions on integer homology $3$--spheres. This list consists of  the finite subgroups of ${\rm SO}(4)$ and the Milnor groups $Q(8n,k,l)$.
The recent results of Perelman imply that  no group of type $Q(8n,k,l)$ acts on $S^3$ \cite{P1,P2}. On the contrary some Milnor groups admit an action on an integer homology $3$--sphere \cite{Mg}.

If we admit arbitrary actions, the list of candidates  is again comparable with the list of finite subgroups of  ${\rm SO}(4)$. For example Reni and Zimmermann (see Zimmermann \cite{Z} and Mecchia and Zimmermann \cite{MZ1}) characterized the nonsolvable groups acting on integer homology $3$--spheres; the unique simple group that admits an action on an  integer homology $3$--sphere is $\A_5$ (and it cannot act freely).  For the standard $3$--sphere,  Thurston's  orbifold geometrization theorem  \cite{BLPo} implies that the finite groups with nonfree actions are exactly the subgroups of  ${\rm SO}(4)$.

On the other hand, Cooper and Long \cite{CL} proved  that  every finite   group admits an  action on a rational homology $3$--sphere (and even a free action).

The  class of $\Bbb{Z}_2$--homology $3$--spheres is intermediate between these two cases. This class is interesting also because  $\Bbb{Z}_2$--homology $3$--spheres appear more frequently than integer homology $3$--spheres; for example  $2$--fold branched coverings of  knots in $S^3$ are $\Bbb{Z}_2$--homology $3$--spheres.   Dotzel and Hamrick \cite{DH} proved that every finite $2$--group acting on a  $\Bbb{Z}_2$--homology $3$--sphere acts orthogonally on $S^3$. This property is not true in general  for solvable groups (already for integer homology $3$--spheres).  In \cite{MZ1} a list of nonsolvable groups which are candidates for actions on $\Bbb{Z}_2$--homology $3$--spheres was given; in this case the only simple groups, that occur, are the projective  special linear groups  ${\rm PSL}(2,q)$. 

In the present paper we  consider finite  groups acting orientation-preservingly on $3$--manifolds  which contain an involution with nonempty connected fixed point set.
We recall that any involution acting on a $\Bbb{Z}_2$--homology $3$--sphere has connected fixed point set (maybe empty), so there are some relations with our situation.
For example the $2$--fold branched coverings of  knots   satisfy both assumptions but in general the two conditions  give different classes of $3$--manifolds.

In fact not all $\Bbb{Z}_2$--homology $3$--spheres admit the action of an involution with nonempty fixed point set. For example if $K$ is a hyperbolic knot in $S^3$  without symmetries, for coefficients sufficiently large, Dehn surgery along the knot gives  a hyperbolic manifold with trivial isometry group (by Thurston's hyperbolic surgery theorem \cite{T}); moreover for $p$ odd  a   $p/q$--surgery  gives a  $\Bbb{Z}_2$--homology $3$--sphere. 

On the other hand  all the $3$--manifolds that are  the $n$--fold cyclic branched covering of a strongly invertible knot admit the  action of an involution with nonempty and connected fixed point set;  it is easy to find examples of $n$--fold cyclic branched coverings of strongly invertible knots that have nontrivial  first  $\Bbb{Z}_2$--homology group (some computation of first  homology group can be found in \cite{Go}). The possibility to study the  $n$--fold cyclic branched coverings of  strongly invertible knots is one of the motivations of this paper.
Another  example of a $3$--manifold admitting  an involution with nonempty connected fixed point set can be obtained by a $3$--component link $L$ admitting a symmetry $t$ with nonempty fixed point set which acts as a reflection on one component while exchanging the remaining two (eg the Borromean rings); the $2$--fold branched covering $M$  of $L$ has nontrivial first $\Bbb{Z}_2$--homology  group (see Sakuma \cite[Sublemma 15.4]{Sa}) and the lift of $t$ is an involution with  the desired property.

When we consider   finite groups acting on $3$--manifolds, the two different  assumptions imply different analyses. 
In fact for   $\Bbb{Z}_2$--homology $3$--spheres we have  some global information about $2$--groups which admit an action. In our case we can control directly only the centralizer of the involution  with nonempty connected fixed point set, thus it is more  difficult to pass to a global description of the group, even in the case of $2$--groups.

A first step in this direction was obtained by Reni and Zimmermann.

\setcounter{thm}{-1}
\begin{thm}\label{thm0}{\rm \cite{RZ1}}\qua
Let $G$ be a finite group of orientation-preserving diffeomorphisms of a closed orientable $3$--manifold; if $G$ contains an involution with nonempty connected fixed point set, then $G$ has sectional $2$--rank at most four (ie every $2$--subgroup is generated by at most four elements).
\end{thm}

In this paper we try to analyze the whole group. We  describe the structure of the group ``up to solvable sections''. The interest for nonsolvable groups is also motivated by geometry. For example,  if  two knots  have the same hyperbolic  cyclic branched covering  $M$ and the isometry group of $M$ is solvable, then it is possible to describe  the relation between the two knots \cite{RZ1}. The problem is not completely solved  if the isometry group is not solvable.

We summarize part of the description  in the following theorem; we recall that a group $E$ is \textit{semisimple} if it is perfect and the factor group of $E$  by its center is a direct product of nonabelian simple groups (see  Suzuki \cite[Chapter 6.6]{S2} or Gorenstein, Lyons and Solomon \cite[p\,16]{GLS1}).

\begin{thm}\label{thm1}
 Let $G$ be a finite group of orientation-preserving diffeomorphisms of a closed orientable $3$--manifold; we denote by $\O(G)$ the maximal normal subgroup of odd order and by   $E$  the maximal semisimple normal  subgroup of $G/\O(G)$.   Suppose that  $G$ contains an involution with nonempty connected fixed point set. 
\begin{enumerate}
\item If the semisimple group $E$ is not trivial, it has at most two components and the factor group of  $G/\O(G)$ by $E$ is solvable. Moreover the factor group of  $E$ by its center is  either a simple group of sectional $2$--rank at most four or  the direct product of two simple groups with sectional $2$--rank at most two. 

\item If $E$ is trivial, there exists a normal subgroup $N$ of $G$ such that $N$ is solvable and $G/N$ is isomorphic to a subgroup of  ${\rm GL}(4,2)$, the general linear group of $4\times 4$ matrices over the finite  field with 2 elements. 
\end{enumerate}
\end{thm}

The simple groups of sectional $2$--rank at most four are classified by the Gorenstein--Harada Theorem \cite[p\,6]{G},  an important part of the classification of finite simple groups.
A well-known part of the classification,  which was proved before then  the Gorenstein--Harada Theorem,  is the classification of finite simple groups of $2$--rank at most two (ie every elementary $2$--subgroup is generated by at most two elements) \cite[p\,6]{G};  obviously sectional $2$--rank at most two implies $2$--rank at most two.

More details are given in \fullref{Section 3} where \fullref{thm1} is proved.  If $E$ is not trivial, the structure of the solvable group $(G/\O(G))/E$ is well understood.  Also in the second case, if we suppose that the group $G$ is not solvable, a short list of candidates for the group $G/N$ can be produced (the nonsolvable subgroups of   ${\rm GL}(4,2)\cong \A_8$ can be easily deduced from \cite{A}). 

In the   study  of cyclic branched coverings of knots, we are mainly interested in the case when the projection  of the involution with nonempty connected fixed point set is contained in $E$, the maximal semisimple  normal  subgroup. Under this condition the list of candidates is much shorter.

\begin{thm}\label{thm2}
Let $G$ be a finite group of orientation-preserving diffeomorphisms of a closed orientable $3$--manifold; we denote by $\O(G)$ the maximal normal subgroup of odd order and by   $E$  the maximal semisimple normal subgroup of $G/\O(G)$. Suppose that $G$ contains an involution $h$ with nonempty and connected fixed point set such that the coset $h\O(G)$ is contained in $E$; then $G/\O(G)$ has a normal subgroup $D$ isomorphic to one of the following groups:
$${\rm PSL}(2,q),  \quad{\rm PSL}(2,q)\times  \Z_2 \quad  \hbox{or} \quad {\rm SL}(2,q) \times_{\Z_2} {\rm SL}(2,q')$$
where $q$ and $q'$ are odd prime powers greater than four.
The factor group $(G/\O(G))/D$ contains, with index at most two, an abelian subgroup of rank at most four.
\end{thm} 

The group ${\rm SL}(2,q)$ is the special linear group of $2\times 2$ matrices of determinant one over the finite Galois field with $q$ elements. The group ${\rm SL}(2,q)$ is a perfect group which has a unique involution;
this involution generates its center $Z$, and the factor  group ${\rm
SL}(2,q)/Z$ is the projective special linear group  ${\rm PSL}(2,q)$ (which is a
simple group for $q\ge 4$).

The group ${\rm SL}(2,q) \times_{\Z_2} {\rm SL}(2,q')$ is a central product where the involutions in the centers of   ${\rm SL}(2,q)$ and $ {\rm SL}(2,q')$ are identified.

\fullref{thm2} is not simply a  specialization of \fullref{thm1}. We have to do some new work to prove properties of $E$, using directly the fact that $E$ contains the projection of  $h$;  we need also more precise information  about  finite simple groups in the Gorenstein--Harada list.

Probably it is possible to exclude  some groups with sectional $2$--rank at most four also in the general case considered in \fullref{thm1}. A possible  approach is to suppose that the special involution is not in $E$ and consider $\Z_2$--extensions of the simple groups in the Gorenstein--Harada list; some  $\Z_2$--extensions   may have again sectional $2$--rank at most four. At the moment we are  not sure if this  approach case by case, that might be  rather technical and long, can produce a relevant reduction of the list of  the possible groups.

As a corollary of \fullref{thm1}  we can consider the case of semisimple groups (see Reni and Zimmermann \cite{RZ1} for the case of simple groups).

\medskip{\bf Corollary}\qua
{\sl Let $G$ be a semisimple finite group of orientation-preserving diffeomorphisms of a closed orientable $3$--manifold. If $G$ contains an involution $h$ with nonempty and connected fixed point set, then $G$ is isomorphic to one of the following groups:
$${\rm PSL}(2,q)  \quad\hbox{or} \quad {\rm SL}(2,q) \times_{\Z_2} {\rm SL}(2,q')$$
where $q$ and $q'$ are odd prime powers greater than four.}

\medskip We focus now on  some  applications. 
We describes first some  results concerning actions of finite groups on homology $3$--spheres. 

Let  $f$ be a  nontrivial orientation-preserving periodic diffeomorphism of a $3$--manifold $M$. We say that $f$  is   \textit{hyperelliptic} if the quotient orbifold $M/f$ has underlying topological space homeomorphic to $S^3$. 

Using the structure of the finite  $2$--subgroups  acting on  $\Bbb{Z}_2$--homology $3$--spheres, Reni~\cite{R} proved that, up to conjugacy, there are at most nine hyperelliptic involutions acting on a hyperbolic $\Bbb{Z}_2$--homology $3$--sphere; we recall that a hyperelliptic involution on a $\Bbb{Z}_2$--homology $3$--sphere has nonempty connected fixed point set. This is equivalent to say that there exist at most nine inequivalent $\pi$--hyperbolic knots with the same $2$--fold branched covering.

Boileau, Paoluzzi and Zimmermann  \cite{BPaZ}  proved that, up to conjugacy,  at most four cyclic groups generated by a hyperelliptic diffeomorphism of odd prime order can act on an irreducible  integer homology $3$--sphere.  Thus an irreducible integer homology $3$--sphere can be the cyclic branched covering with odd prime order of at most four inequivalent knots. Also in this case a hyperelliptic diffeomorphism of prime order has nonempty connected fixed point set. The characterization of the finite nonsolvable groups which act on integer homology $3$--spheres   plays an important role in the proof in the  hyperbolic case. We remark that one of the basic steps in the proof of the upper  bound is the fact that hyperelliptic diffeomorphisms often commute  and   nonabelian situations are, in some sense, exceptions that can be  described.

The commutativity of hyperelliptic diffeomorphisms  corresponds in the language of knots  to the standard abelian construction. 

\medskip{\bf The standard abelian construction}\qua   
Suppose $M$ is the $n$--fold and $m$--fold cyclic branched covering of two knots $K$ and $K^{\prime}$, respectively.
We denote by $H$ and $H'$ the cyclic transformation groups of $K$ and $K^{\prime }$, respectively; the preimage $\tilde K$ (resp.\ $\tilde K^{\prime} $)  of $K$  (resp.\ $K^{\prime}$) in $M$ is  the fixed point set of $H$ (resp.\ $H'$).
 The groups $H$ and $\smash{H'}$ commute  and they generate a group $A$  of diffeomorphisms of $M$  isomorphic to $\Bbb{Z}_n\times \Bbb{Z}_m$; when $n=m$ the group $A$ has rank two and it is isomorphic to   $\Bbb{Z}_n\times \Bbb{Z}_n$.
Each element of the transformation group $H$ (resp.\ $H'$) induces a rotation on $\smash{\tilde K ^{\prime}}$ (resp.\ $\smash{\tilde K}$), 
and the quotient orbifold $M/A$ is the $3$--sphere whose  singular   set is a link $L=\bar K \cup \bar K'$, where $\bar K$ (resp.\ $\bar K'$) is the projection of $K$ (resp.\ $K'$).

We remark that by  the positive solution to the Smith Conjecture both  components of $L$  are trivial knots. On the other hand, starting from $L$, we can obtain $K$ (resp.\ $\smash{K'}$) taking  the preimage of $\bar  K$ (resp.\ $\bar K'$) in the  $m$--fold (resp.\ $n$--fold) cyclic branched covering  of $\bar K'$ (resp.\ $\bar K$).
This construction serves to study the  relation between two links with the same hyperbolic cyclic branched covering (see  Reni and Zimmermann \cite{RZ2} and  Mecchia \cite{M}).   The standard abelian construction is the unique possibility in many different situations.

\begin{thmn}{\rm \cite{RZ2}}\qua
Let $M$ be a hyperbolic $3$--manifold. Suppose that $M$  is the $n$--fold and $m$--fold cyclic branched covering of inequivalent  knots $K$ and $K'$, respectively, such that $m$ and $n$ are  not  powers of two. Suppose that one of the following conditions holds:
\begin{enumerate}
 \item   $n$ and $m$ have a common prime divisor different from two;

 \item  $K$ is not  strongly invertible and $K$ is  not self-symmetric with order $n$;

 \item  The orientation-preserving isometry group of $M$ is solvable.
\end{enumerate}    
Then $K$ and $K'$  arise from the standard  abelian construction.
\end{thmn}

A $2$--component link is called \textit{ symmetric} if  there exists an orientation-preserving diffeomorphism of $S^3$ which exchanges the $2$--components of the link.
A \textit{cyclic symmetry} of a knot $K$ is a  diffeomorphism of $(S^3,K)$ of finite order  and  with nonempty fixed point set $F$ disjoint from $K$. The set  $F$ is an unknotted circle by the positive solution to the Smith Conjecture. The quotient of $S^3$ by a cyclic symmetry is again the $3$--sphere and $F$ and $K$ project to  a $2$--component link.
We call a knot $K$ \textit{self-symmetric} with order $n$ if $K$ admits a cyclic symmetry $f$ of order $n$ such that the associated quotient link  is symmetric.

We use \fullref{thm2}  to generalize  point 2 of the previous Theorem. We want to include also  the class of strongly invertible knots that is largely studied in knot theory. Unfortunately the standard  abelian construction does not remain the unique possibility.

\begin{thm}\label{thm3}
 Let $M$ be a hyperbolic $3$--manifold. Suppose that $M$ is  the $n$--fold and $m$--fold cyclic branched covering of two hyperbolic knots $K$ and $K'$, respectively, such that $m$ and $n$ are not  powers of two. Let  $G$ be the orientation-preserving isometry group of $M$ and $\O(G)$ the maximal normal subgroup of odd order. If  the knot $K$  is not self-symmetric with order $n$, then one of the following cases occurs:
\begin{enumerate}
\item $K$ and $K'$ arise from the standard  abelian construction;

\item $G$ contains  $h$,  an involution with nonempty connected fixed point set, such that $h\O(G)$ is contained in the maximal normal semisimple subgroup of $G/\O(G)$ (in particular \fullref{thm2}  applies to $G$);

\item All  prime divisors of  $n$ and  $m$ are  contained in $\{2,3,5,7\}$ and   there exists a normal subgroup $N$ of $G$ such that $N$ is solvable and $G/N$ is isomorphic to a subgroup of  ${\rm GL}(4,2)$.  
\end{enumerate}
The   knots $K$ and $K'$ in \fullref{thm3} are inequivalent. It follows from volume considerations if $n\neq m$,  and from the fact that $K$ is not self-symmetric if $n=m$. 
\end{thm}
As in the case of integer homology $3$--spheres, the noncommuting situations are, in some sense, exceptional.  For integer homology $3$--spheres there exists an universal bound to the number of cyclic groups generated by a  hyperelliptic diffeomorphism (with connected nonempty fixed point set) of odd prime order; we propose the following:

\medskip
{\bf  Conjecture}\qua
There exists a universal bound $C$ such that any hyperbolic orientable closed $3$--manifold admits  at most $C$ nonconjugate cyclic groups generated by a  hyperelliptic diffeomorphism  with connected nonempty fixed point set.

\medskip We remark that the condition about the fixed point set is necessary; in general, there is no universal bound for hyperelliptic diffeomorphisms in hyperbolic $3$--manifolds~\cite{RZ3}. We recall that Cooper and Long \cite{CL} proved  that  every finite   group admits an action on a hyperbolic rational homology $3$--sphere; to prove the conjecture the use of  homology may be insufficient.  Probably  we have to consider directly conditions   about  the fixed point sets of the diffeomorphisms, for example the existence of involutions with connected nonempty fixed point set (the hypothesis considered  in this paper).

\section{Preliminary results}\label{Section 2}
In this section we present some preliminary results  about finite groups acting on $3$--manifolds.

\begin{prop}\label{prop1}
Let $G$ be a finite group of orientation-preserving diffeomorphisms of a
closed orientable $3$--manifold and $f$ an element in $G$  with nonempty connected
fixed point set $K$. Then the normalizer $N_G (f)$ of the subgroup generated by $f$ in
$G$ is isomorphic to a subgroup of a semidirect product
$${\Bbb Z}_2 \ltimes ({\Bbb Z}_a\times {\Bbb Z}_b),$$ for some nonnegative
integers $a$ and
$b$, where a generator of ${\Bbb Z}_2$ (an
$f$--reflection, ie acting as a reflection on $K$ ) acts on the normal subgroup ${\Bbb Z}_a \times {\Bbb Z}_b$ of
$f$--rotations (ie the elements  acting  as rotations on $K$) by sending each element to its inverse. In particular, $N_G(f)$ is
solvable.
\end{prop}

\begin{proof}  See Mecchia and Zimmermann \cite[Lemma 1]{MZ2}.\end{proof}

\begin{prop}\label{prop2}
Let $G$ be a finite group of orientation-preserving diffeomorphisms of a
closed orientable $3$--manifold. If $G$ is  isomorphic to ${\Bbb Z}_2 \times {\Bbb Z}_2 \times {\Bbb Z}_2$,  there exists in $G$ an  involution  that  either acts freely or has nonconnected fixed point set.
\end{prop}

\begin{proof} By contradiction we suppose that the seven involutions in $G$ have connected and nonempty fixed point set. Let $f$ be one of the involutions, by \fullref{prop1} the group  $G$ contains  four $f$--rotations and four $f$--reflections. We denote by $r$ one $f$--rotation of order two different from $f$ and we denote by $t$ one $f$--reflection, so the four $f$ rotations are Id, $f,\,r$ and $rf$ and the four $f$ reflections are $t,\,tf,\,tr$ and $tfr$.
Since the fixed point sets of $f$ and $t$ have nonempty intersection then the fixed point sets  of $t$ and $tf$ have nonempty intersection, then $t$ is a $tf$--reflection.
Now we consider $r$ and $rf$, both of them have nonempty connected fixed point set and both the subgroups of $r$ and $rf$--rotations coincide with the subgroup of $f$--rotations. We deduce that $t$ is an $r$--reflection and an $rf$--reflection and consequently $t$ is an $rt$--reflection and an $rft$--reflection.
It turns out that $t$ acts  as a reflection on the fixed point set of each involution in $G$ different from $t$ and viceversa each involution in $G$ different from $t$ acts as a reflection on the fixed point set of $t$;  so in $G$ we have six $t$--reflections and this is impossible by \fullref{prop1}.
\end{proof}
We conclude the section with a purely algebraic  proposition that describes   the centralizer  of an involution in the factor groups by odd order normal subgroups. This  proposition shows that, for an involution $t$, the quotient of the centralizer is the centralizer of the projection of $t$ in the quotient; its proof is elementary but we  often  use this fact.

\begin{prop}\label{prop3}
Let $t$  be an involution in a  finite group $G$ and let  $N$ be  a normal subgroup of $G$ of odd order. Then the centralizer $C_{G/N}(tN)$ of the coset $tN$ in $G/N$ is isomorphic to $C_{G}(t)/(C_{G}(t)\cap N),$ that is the factor group of the centralizer of $t$ in $G$ by the intersection  $C_{G}(t)\cap N$.
\end{prop}

\begin{proof} Indeed we prove the equality  $\{cN|c\in  C_{G}(t)\}=C_{G/N}(tN)$ and  then the thesis follows from the Second Isomorphism Theorem.

The inclusion  $\{cN|c\in  C_{G}(t)\}\subseteq C_{G/N}(tN)$ is trivial.

We suppose that $fN$ is contained in  $C_{G/N}(tN)$ that is $ftf^{-1}N=tN$, so there exists $k\in N$ such that $ftf^{-1}=tk$. The subgroup $\langle t,N \rangle $ of $G$  generated by $t$ and $N$  has a Sylow $2$--subgroup of order two, so  all the involutions in $\langle t,N\rangle$ are conjugate, in particular  there exists an element $g\in N$ such that $gtg^{-1}=tk$. It follows that $g^{-1}f$ is contained in  $C_{G}(t)$ and, since $g\in N$, we have that   $f$ is contained in   $fN=Nf=N(g^{-1}f)=(g^{-1}f)N$. The coset $fN$ is contained in $\{cN|c\in  C_{G}(t)\}$ and the inclusion $\{cN|c\in  C_{G}(t)\}\supseteq C_{G/N}(tN)$ is proved.
\end{proof}

\section[Proof of \ref{thm1}]{Proof of \fullref{thm1}}\label{Section 3}

We denote by $\bar G$ the factor group $G/\O(G)$ and by $\tilde E$ the factor group of $E$ by its center $Z(E)$.

\medskip{\bf Step 1}\qua The maximal semisimple normal subgroup $E$ has sectional $2$--rank at most four and it has at most two components.  If $E$ has two components, $\tilde E$ is the direct product of two simple groups with sectional $2$--rank two.

By \fullref{thm0}, $E$ has sectional $2$--rank at most four and consequently $\tilde E$ has  sectional $2$--rank at most four.  
We recall that a minimal set of generators of a group means a set of generators such  that  any proper subset does not generate the group. In general we can have minimal sets of generators with different numbers of elements for the same finite group but,  by Burnside's basis theorem \cite[Theorem 1.16, p\,92]{S1}, any two  minimal sets of generators of a $p$--group contain the same number of elements.

Moreover, in   the direct product of two groups, the union of a minimal set of generators of the first group   with  a minimal set of generators of the second group is a minimal set of generators of the direct product. It follows that the sectional $2$--rank of the direct product of two groups is equal or  greater then the sum of the sectional $2$--ranks of the two direct factors.  Since simple groups have sectional $2$--rank at least two  \cite[p.144]{S2} we get the thesis. 

\medskip{\bf Step 2}\qua  We denote  by $C$ the centralizer $C_{\bar G}(E)$ of $E$ in $\bar G$.   If  $E$ is not trivial, then $C$ is solvable.

Since $C$ is the centralizer of a normal subgroup, $C$ is normal in $G$.
The intersection of $C$ and $E$ is $Z(E)$, the center of $E$. The center of $E$ has order a power of two, otherwise $\O(G)$ is not maximal.   We denote by  $D$  the group generated by $C$ and $E$;  the group  $D$ is a central product of $E$ and $C$. 
By \fullref{thm0}, the sectional $2$--rank of $D$ is equal or smaller  then four; it follows that $D/Z(E)$ has sectional $2$--rank  equal or smaller then four. The factor group $D/Z(E)$ is isomorphic to $E/Z(E)\times C/Z(E)$; the sectional $2$--rank of $E/Z(E)$ is at least two, so the sectional $2$--rank of $C/Z(E)$ is at most two (see Step 1).

The maximal semisimple normal subgroup of $C$ is trivial, otherwise $E$ is not maximal. We consider $F(C)$ the generalized Fitting subgroup of $C$. We recall that  the generalized Fitting subgroup is the subgroup  generated by the maximal semisimple normal subgroup   and by the Fitting subgroup; the Fitting subgroup is the maximal nilpotent normal subgroup \cite[p\,452]{S2}. In this case, since the maximal semisimple normal subgroup of $C$ is trivial, $F(C)$ coincides with the Fitting subgroup. 
Note that, since $F(C)$ is nilpotent, its Hall subgroup of maximal odd order is unique. Since $F(C)$ is characteristic in $C$, the generalized Fitting subgroup $F(C)$ is a $2$--group, otherwise $\O(G)$ is not maximal. The group $C$ acts on $F(C)$ by conjugation.  The centralizer $C_C(F(C))$ of $F(C)$ in $C$ is contained in $F(C)$ \cite[Theorem 6.11, p\,452]{S2} and in particular it is a $2$--group;  the factor group $C/C_C(F(C))$ is a subgroup of the  automorphism group of $F(C)$. Let $\Phi$ be the Frattini  subgroup of $F(C)$; the factor group $F(C)/\Phi$ is an elementary  abelian group.  The totality of automorphisms that leave every element of $F(C)/\Phi$ invariant is a normal $2$--subgroup of   ${\rm Aut}(F(C))$  \mbox{\cite[Theorem 1.17, p\,93]{S1}}. Let $T$ be the subgroup of $C$ of elements that act trivially on $F(C)/\Phi$; then $T$ is a normal $2$--subgroup  and $C/T$ is a subgroup of  ${\rm GL}(d,2)$, where $d$ is the rank of $F(C)/\Phi$.
Since $F$ has sectional $2$--rank at most four we have $d\leq 4$; if $d\leq 2$ the group  ${\rm GL}(d,2)$ is solvable and the proof is finished.

Suppose that $d=3$. The group ${\rm GL}(3,2)$ has order $2^3\cdot 3 \cdot 7$; any automorphism of order seven permutes cyclically all the involutions in ${\Bbb Z}_2\times {\Bbb Z}_2\times{\Bbb Z}_2 $. In this case $\Phi$ cannot contain $Z(E)$ because $F(C)/Z(E)$ must have sectional $2$--rank at most two. At least one involution in  $F(C)/\Phi$ is the projection of an element in $Z(E)$ and it  is contained in the center of $C$;  this involution is fixed by conjugation by each element of $C$ and $C/T$ cannot contain any element of order seven; $C/T$ has order at most 24 and it is solvable. 

Suppose finally that $d=4$. The group ${\rm GL}(4,2)$ has order $2^6\cdot 3^2 \cdot 5\cdot 7$;  an automorphism of order five does not centralize any involution of
 ${\Bbb Z}_2\times {\Bbb Z}_2\times{\Bbb Z}_2 \times{\Bbb Z}_2 $ (we have three orbits with five elements) and an automorphism of order
seven centralizes exactly one involution (two orbits with seven elements and one orbit with only one element).
We consider the group $Z(E)\cdot \Phi$ generated by $Z(E)$ and $\Phi$. The group $F(C)/(Z(E)\cdot \Phi)$ must have rank at most two. It follows that at least three involutions in $F(C)/\Phi$ are projections of elements in the center of $E$. The group $C/T$ cannot contain elements of order five or seven and hence the order of   $C/T$ is product of powers of 2 and 3.  By Burnside's Theorem \cite[Theorem 4.25, p\,216]{S2}, any such group is solvable. 

This finishes the proof of Step 2.

\medskip{\bf Step 3}\qua If    $E$ is not trivial, $\bar G/E$ is solvable.

The normal subgroup $D$ is the subgroup generated by $E$ and
$C=C_{\bar G}(E)$; we consider the factor group $\bar G/D$ that is
isomorphic to a subgroup of Out$(E)$, the outer automorphism group of
$E$.  If an automorphism of $E$ acts trivially on $\tilde E=E/Z(E)$,
it acts trivially on $E$; this is a consequence of the three subgroups
lemma \cite[(6.3), p\,447]{S2}, \cite[Lemma 3.8, p\,7]{GLS2} and of
the fact that $E$ is perfect. It follows that the group Out($E$) is a
subgroup of Out($\tilde E$).

We recall that the outer automorphism group  of a simple group    is solvable (for a discussion about this  property, called the Schreier property,  see  \cite[p\,4]{GLS2}).

The group $\tilde E$ is either a simple group with sectional $2$--rank at most four or the direct product of two simple groups  with sectional $2$--rank at most two; in this last  case  Out($\tilde E$) contains, with index at most two, the direct product of the outer automorphism groups of the two components \cite[Lemma 3.23, p\,13]{GLS2}.
In any case  Out($\tilde E$) is solvable; it follows that $\bar G/D$ and hence  $\bar G/E$ are solvable.

\medskip{\bf Step 4}\qua If $E$ is trivial,  there exists a normal subgroup $N$ of $G$ such that $N$ is solvable and $G/N$ is isomorphic to a subgroup of  ${\rm GL}(4,2)$.

We consider $F(\bar G)$ the generalized Fitting subgroup of $\bar G$; since $E$ is trivial,   $F(\bar G)$ coincides with the Fitting subgroup. 
The subgroup $F(\bar G)$ does not contain any element with odd order, otherwise $\O(G)$ is not maximal.
The generalized Fitting subgroup contains $C_{\bar G}(F(\bar G))$ its centralizer in $\bar G$ \cite[Theorem 6.11, p\,452]{S2} and in particular $C_{\bar G}(F(\bar G))$  is a $2$--group; the factor group of $\bar G$ by  $C_{\bar G}(F(\bar G))$ is isomorphic to a subgroup of  ${\rm Aut}(F(\bar G))$, the automorphism group of $F(\bar G)$.

  We consider $\Phi$, the Frattini subgroup of $F(\bar G)$. As a consequence, the factor group $F(\bar G)/\Phi$ is an elementary group. The totality of automorphisms that leave every element of $F(\bar G)/\Phi$ invariant is a normal $2$--subgroup of   ${\rm Aut}(F(\bar G))$  \cite[\mbox{Theorem 1.17}, p\,93]{S1}.
The factor group $\bar G$ contains $\bar N$,  a normal $2$--subgroup, such that $\bar G/\bar N$ is  isomorphic to a subgroup of    ${\rm GL}(d,2)$ where $d$ is the rank of $F(\bar G)/\Phi$.  
We denote by $N$ the preimage of $\bar N$ with respect to  the projection of $G$ onto  $\bar G$; we remark that $G/N$ is  isomorphic to a subgroup of  ${\rm GL}(d,2)$ and $N/\O(G)$ is a $2$--group.
If $G$ contains an involution with nonempty connected fixed point set, \fullref{thm0}  implies that  $G$ has sectional $2$--rank at most four and  hence we can set  $d=4$.

\section[Proof of \ref{thm2}]{Proof of \fullref{thm2}}\label{Section 4}

To simplify the notation  we denote  by $\bar G$ the factor group   $G/\O(G)$ and we denote by $\bar g$ the coset  $gO(G)$ where $g$ is an element of $G$; by hypothesis we have an involution $h$ in $G$ with connected and nonempty fixed point set such that $\bar h$ is contained in $E$, the maximal semisimple normal subgroup of $\bar G$.

In the proof we  often use the following fact: by \fullref{prop1} and \fullref{prop3},  if we have an involution $t$ in $G$ with nonempty fixed point set,   the centralizer $C_{\bar G}(\bar t)$ of $\bar t$ is isomorphic to subgroup of  a semidirect product ${\Bbb Z}_2 \ltimes ({\Bbb Z}_a\times {\Bbb Z}_b)$ where a generator of ${\Bbb Z}_2$  acts on the normal subgroup ${\Bbb Z}_a \times {\Bbb Z}_b$  by sending each element to its inverse. In particular  $C_{\bar G}(\bar h)$ is isomorphic to  $C_{G}(h)/(C_{G}(h)\cap \O(G))$ and we call $\bar h$--rotations  (resp.\ $\bar h$--reflections) the elements of  $C_{\bar G}(\bar h)$ that are projections of $h$--rotations  (resp.\ $h$--reflections); since  $\O(G)$ cannot contain $h$--reflections, this notation is not ambiguous.

In the proof we call  a group \textit{admissible} if it has a subgroup of index at most two that  is  isomorphic to a subgroup of  a semidirect product ${\Bbb Z}_2 \ltimes ({\Bbb Z}_a\times {\Bbb Z}_b)$, where a generator of ${\Bbb Z}_2$  acts on the normal subgroup ${\Bbb Z}_a \times {\Bbb Z}_b$  by sending each element to its inverse. We note that an admissible group is solvable and  a subgroup or a factor group of an admissible group is again admissible. We remark also that subgroups of $\bar G$, that contain the centralizer $C_{\bar G}(\bar h)$  with index at most two, are admissible.

\medskip
{\bf Step 1}\qua  The order of $Z(E)$, the  center of $E$,  is a power of two, the involution $\bar h$ is not contained in $Z(E)$ and  either $Z(E)$ is cyclic or   $Z(E)$ is elementary abelian of order four and $C_E (\bar h)$ is elementary abelian of order eight.

The order  of   $Z(E)$ is  a power of two, otherwise $\O(G)$ is not maximal.
Since   $C_E (\bar h)$ is solvable,  the  center  $Z(E)$ does not contain $\bar h$. Since $\bar h \in E$, the center  $Z(E)$ is a subgroup of   $C_E (\bar h)$, that  is isomorphic to a subgroup of  the semidirect product ${\Bbb Z}_2 \ltimes ({\Bbb Z}_{2^n}\times {\Bbb Z}_{2^m})$. 
If  $Z(E)$ contains an element with order strictly greater than two, then $Z(E)$ can contain only $\bar h$--rotations and, since $\bar h  \notin Z(E)$, the center $Z(E)$ can contain only one involution; this fact implies that $Z(E)$ has to be cyclic.

Suppose now that all the nontrivial elements in   $Z(E)$ have order two. If $Z(E)$ is not cyclic, it must contain an $\bar h$--reflection which thus must commute with the whole group and we have only one possibility:   $Z(E)\cong {\Bbb Z}_{2}\times {\Bbb Z}_{2}$ and  $C_E (\bar h)\cong {\Bbb Z}_{2}\times {\Bbb Z}_{2}\times {\Bbb Z}_{2}$.

\medskip{\bf Step 2}\qua We denote by $\tilde E$ the factor group $E/Z(E)$ and we denote by $\tilde h$ the coset $\bar h Z(E)$. We consider  $C_{\smash{\tilde E}} (\tilde h)$  the centralizer of $\tilde h$ in $\tilde E$. 
\begin{enumerate}
\item   If $Z(E)$ is not cyclic, then  $C_{\smash{\tilde E}} (\tilde h)$ has order at most eight.
\item  If $Z(E)$ is cyclic, then $C_{\smash{\tilde E}} (\tilde h)$ contains with index at most two the factor group  $C_E (\bar h)/Z(E)$.  
\end{enumerate}
In both  cases  $C_{\smash{\tilde E}} (\tilde h)$ is admissible.

We denote by $P$ the subgroup $\{\bar f\in E|\,\exists \, \bar g\in Z(E) \,\,\, {\rm such} \,\,\, {\rm  that } \,\,\, \bar f\bar h \bar f^{-1}=\bar h \bar g \}$
that is the preimage of $C_{\smash{\tilde E}} (\tilde h)$  with respect to  the standard projection of $E$ onto $\tilde E=E/Z(E)$; we recall  that $\smash{C_{\smash{\tilde E}} (\tilde h)=P/Z(E)}$.

If $Z(E)$ is the trivial group the thesis trivially holds. 

We suppose that  $Z(E)$ is cyclic and nontrivial; we denote by $\bar z$ the unique involution in  $Z(E)$ and we get the following equality:
 $$P=\{\bar f\in E|\,\,\, {\rm either} \,\,\, \bar f\bar h \bar f^{-1}=\bar h \,\,\, {\rm or} \,\,\, \bar f\bar h \bar f^{-1}=\bar h \bar z \}.$$
In this case we obtain that $C_{\smash{\tilde E}} (\tilde h)$ contains with index at most two the factor group $C_E (\bar h)/Z(E)$. 

Finally we suppose that  $Z(E)$ is not cyclic. By Step 1, we have that $Z(E)\cong {\Bbb Z}_{2}\times {\Bbb Z}_{2}$ and  $C_E (\bar h)\cong {\Bbb Z}_{2}\times {\Bbb Z}_{2}\times {\Bbb Z}_{2}$ for the center contains an $\bar h$--reflection. The centralizer $C_E(\bar h)$ is a normal subgroup of $P$; since $C_E(\bar h)$ contains its centralizer in $P$, the factor group $P/C_E(\bar h)$  acts effectively on $C_E(\bar h)$ by conjugation and $P/C_E(\bar h)$ is isomorphic to a subgroup  of   Aut  $({\Bbb Z}_{2}\times {\Bbb Z}_{2}\times {\Bbb Z}_{2})$, the automorphism group of the elementary abelian group of order eight.
Moreover $P/C_E(\bar h)$ leaves invariant elementwise $Z(E)\cong {\Bbb Z}_{2}\times {\Bbb Z}_{2}$ that is a subgroup of index two in  $C_E(\bar h)$; this fact implies that  $P/C_E(\bar h)$ is a subgroup of  ${\Bbb Z}_{2}\times {\Bbb Z}_{2}$. So  $C_{\smash{\tilde E}} (\tilde h)=P/Z(E).$ has order at most eight. This finishes the proof of Step 2.

\medskip{\bf Step 3}\qua If $E$ has one component, $\tilde E$ has only one conjugacy class of involutions.

In this case $\tilde E$ is a simple group with sectional $2$--rank at most four and we apply the Gorenstein--Harada classification of finite simple groups of sectional $2$--rank at most four (see Gorenstein \cite[p\,6]{G} and Suzuki \cite[Theorem 8.12, p\,513]{S2}).

The group $\tilde E$ contains $\tilde h$ and the centralizer of $\tilde h$ is admissible. We will show that no group in the Gorenstein--Harada list which has more than one conjugacy class of involutions contains an involution with an admissible centralizer.

The following list of  groups contains all the simple groups with sectional $2$--rank at most four and more than one conjugacy class of involutions (the algebraic properties of the simple groups can be found in the \emph{Atlas of finite groups} \cite{A}, Sakuma \cite[Chapter~6.5]{S2} or Gorenstein \cite{G}):
$$\begin{array}{l}M_{12};\ {\rm PSp}(4,q), \hbox{for } q \hbox{ odd};\ J_2;\
\Bbb A_n,  \hbox{for } 8\leq n\leq 11;\\ {\rm PSL} (4,q),\
{\rm PSU} (4,q),\ {\rm PSL} (5,q) \hbox{ and }{\rm PSU} (5,q), \hbox{for } q \hbox{ odd}.\end{array}$$
We can rule out directly the following groups  because the centralizer of any
involution in these groups is not solvable (for the groups of Lie type see Suzuki \cite[6.5.2, 6.5.7, 6.5.15]{S2}; for $J_2$  see Gorenstein \cite[p\,99]{G} or the \emph{Atlas of finite groups} \cite{A}):
$$\begin{array}{l}J_2;\ {\rm PSp}(4,q);\  {\rm PSL} (4,q) \hbox{ and }{\rm PSU} (4,q), \hbox{ for }q\hbox{ odd, }
q\geq 5;\\ {\rm PSL} (5,q)\hbox{ and  }{\rm PSU} (5,q),\hbox{ for }q\hbox{ odd.}\end{array}$$
 The Mathieu group $M_{12}$ and   the alternating groups $\Bbb A_n$, for $8\leq n\leq 11$, contain some involutions with solvable centralizer  but the centralizers of such involutions  contain $\Bbb S_4$, that is not admissible (\cite{A} for $M_{12}$).

Finally, in the groups ${\rm PSp}(4,3)$,  ${\rm PSL} (4,3)$ and  ${\rm PSU} (4,3)$,   the centralizer of each involution contains a subgroup with a factor group isomorphic to the non admissible group $\A_4 \cong {\rm PSL} (2,3) \cong {\rm PSU} (2,3)$  \cite[6.5.2, 6.5.7, 6.5.15]{S2}. 
This concludes the proof.

\medskip{\bf Step 4}\qua We denote by $\tilde S_2$ a Sylow $2$--subgroup  of  $\tilde E$; if $E$ has one component  either $\tilde S_2$ has sectional $2$--rank two or  $\tilde S_2$ is an elementary abelian group with eight elements.

Since by Step 3 the involutions in $\tilde E$ are all conjugate, we can suppose that $\tilde h$ is central in $\tilde S_2$, this implies that $\smash{\tilde S_2=C_{\smash{\tilde S_2}}(\tilde h)}$. We denote by ${\cal E}$ the preimage of $E$ in $G$ with respect to  the projection of $G$ onto $\bar G$.  We recall that we described  $Z(E)$ in Step 1; we consider three cases according to the structure of $Z(E)$. We remark also that  a $2$--group  with order at most eight which is not elementary abelian of rank three, has sectional $2$--rank  at most two, so when we will obtain  that  $\tilde S_2$ has  order at most eight, we will get the thesis. 

Suppose first that $Z(E)$ is trivial. In this case $\tilde S_2$ is isomorphic to the Sylow $2$--subgroup of ${\cal E}$. The involutions in  ${\cal E}$ are all conjugate. In fact if we consider $t$ and $t'$ two involutions in ${\cal E}$, we know that $\bar t$ and $\bar t'$ are conjugate in $E$, so there exists $g$ in $\O(G)$  such that $t$ is conjugate to $t'g$. Since $\O(G)$ has odd order, the group  generated by $t'$ and $\O(G)$ has Sylow $2$--subgroup of order two and all the involutions in the group are conjugate; in particular $t'$ and $t'g$ are conjugate. We can conclude that $t$ and $t'$ are conjugate.  
All the involutions in ${\cal E}$ are conjugate to $h$, so all the involutions have nonempty connected fixed point set; by \fullref{prop2}, the group  ${\cal E}$ cannot contain a subgroup isomorphic to  ${\Bbb Z}_2 \times {\Bbb Z}_2 \times {\Bbb Z}_2$. Since, by \fullref{prop1},  the  Sylow $2$--subgroup of  ${\cal E}$ is a subgroup of the semidirect product ${\Bbb Z}_2 \ltimes ({\Bbb Z}_{2^a}\times {\Bbb Z}_{2^b})$, we obtain that $\tilde S_2$  is dihedral or abelian of rank two. 

If $Z(E)$ is elementary abelian of order four, by Step 2  we have that $\tilde S_2=C_{\tilde S_2} (\tilde h)$ has order at most eight and we get thesis.

Finally we suppose that  $Z(E)$ is cyclic and nontrivial.
We consider $S_2$ the Sylow $2$--subgroup of $E$, the center  $Z(E)$ is contained in $S_2$ and $\tilde S_2$ is the projection of $S_2$. 
By Step 2 we  can assume  that $(C_{S_2}(\bar h)/Z(E))$ has index at most two in $\tilde S_2$.  

If $Z(E)$ contains an $\bar h$--reflection, by \fullref{prop1} the centralizer $C_{E}(\bar h)$ has order at most eight, and we conclude that  $(C_{S_2}(\bar h)/Z(E))$  has order at most four and  $\tilde S_2$ has order at most eight.

So we can suppose that $Z(E)$ contains only $\bar h$--rotations. We denote by  $R$  the subgroup  of $\bar h$--rotations contained in the Sylow $2$--subgroup of $E$; the subgroup $R$  contains $Z(E)$. 

We obtain that the factor group $R/Z(E)$ is cyclic. In fact, if  $R/Z(E)$ has rank two, we have an $\bar h$--rotation  $\bar f$ of order different than two such that $\bar f\notin Z(E)$ and $\bar f^2\in Z(E)$. The coset $\bar f Z(E)$ contains no involution and the coset  $\bar hZ(E)$ contains two involutions for $\bar h$ is not in the center; on the other hand since   $\bar f Z(E)$  and $\bar hZ(E)$ represent two involutions in $\tilde E$, they are conjugate and this gives a contradiction.

This concludes the proof in the case that $(C_{S_2}(\bar h)/Z(E))=\tilde S_2$.

On the other hand,  if $C_{S_2}(\bar h)$ does not contain any $\bar h$--reflection, $\tilde S_2$ contains a cyclic subgroup of index at most two, so  it has sectional $2$--rank at most two and the proof is finished. 

So we can suppose the following two facts:

\begin{enumerate}
\item   $C_{S_2}(\bar h)$ contains $\bar t$ an $\bar h$--reflection; 
\item   $(C_{S_2}(\bar h)/Z(E))$ has index two in $\tilde S_2=C_{\tilde S_2} (\tilde h)$; in this case there exist two nontrivial elements $\bar s$ in $S_2$ and $\bar c$ in $Z(E)$  such that $\smash{\bar s \bar h\bar s^{-1}=\bar h \bar c}$.
\end{enumerate}

Since $\bar t Z(E)$ is conjugate to $\bar hZ(E)$ and $\bar t Z(E)$ contains a number of involutions equal to the order of $Z(E)$ (these elements are all reflections), we obtain that $Z(E)$ has order two. Since $R/Z(E)$ is cyclic, we have  $R\cong  {\Bbb Z}_2 \times {\Bbb Z}_{2^m}$. Moreover we obtain that $m=1$; in fact, if  $R$ contains an element of order strictly greater than two,  one involution between  $\bar h$ and  $\bar h \bar c$ is characteristic in   $C_{S_2}(\bar h)$  (all the involutions which are obtained as powers of elements of order strictly greater than two coincide)  and this is in contradiction with the existence of $\bar s$. We can conclude that  $C_{S_2}(\bar h)$ is an elementary group of order eight  and $\tilde S_2$ has order eight.

\medskip{\bf Step 5}\qua If $E$ has one component, $E$ is isomorphic to ${\rm PSL}(2,q)$, with $q\geq 5$.

By Step 3  the simple group $\tilde E=E/Z(E)$ has only one conjugacy class of involutions and by Step 4   $\tilde E$  has sectional $2$--rank smaller than two or $\tilde S_2$ is an elementary abelian group of order eight.
In the  Gorenstein--Harada list  we find the following groups that satisfy these properties and that were not already excluded in Step 3:
$$\begin{array}{l}{\rm PSL} (2,q),\hbox{ for }q\hbox{ odd and }q\geq 5;\ {\rm PSL} (3,q) \hbox{ and  }{\rm PSU} (3,q) \hbox{ for }q\hbox{ odd;}\\ M_{11};
\Bbb A_7;\    J_1;\ {\rm PSL} (2,8);\ ^{2}G_2(3^n), \hbox{ for }n>1.
\end{array}$$
We recall that $\tilde E$ has to contain $\tilde h$ an involution with admissible centralizer. 

In the  groups  $J_1$ and   $^{2}G_2(3^n)$ for $n>1$   the centralizer of an involution is isomorphic to the non admissible group $ {\Bbb Z}_2 \times \Bbb{\rm PSL} (2,q)$ with $q>5$ \cite[p\,514]{S2}.  

We can rule out ${\rm PSL} (3,q)$ and  ${\rm PSU} (3,q)$  for $q$ odd,  $q\geq 5$,  because the centralizer of any
involution in these groups is not solvable \cite[6.5.2, 6.5.15]{S2}.

We consider the groups ${\rm PSL} (3,3)$, ${\rm PSU} (3,3)$ and
$M_{11}$. The centralizer of any involution in these groups has a
subgroup which has the alternating group $\A_4 \cong {\rm PSL} (2,3)
\cong {\rm PSU} (2,3)$ as a factor group and so it is not admissible
\cite[6.5.2, 6.5.15]{S2}, \cite{A}.

The group $PSL(2,8)$ does not admit central perfect extension \cite{A}; in this case $Z(E)$ should be trivial. 
The Sylow $2$--subgroup of  $PSL(2,8)$ is elementary abelian of order eight,  
by the same argument used in Step 4 for the case of  $Z(E)$ trivial we can exclude this group.  

We consider $\Bbb A_7$. If $Z(E)$ is not trivial, the unique central extension of  $\Bbb A_7$ with center of  order a power of two  is $\A^*_7$. The  Sylow $2$--subgroup of $\A^*_7$ is a quaternion group of order eight and it contains a unique involution that is central  in the group and this is impossible. 
We can suppose that $Z(E)$ is trivial and $E\cong \Bbb A_7$. We consider  the  centralizer of the involution $\bar h$ in $ \Bbb A_7$; we can suppose up to conjugation that  $\bar h$  is the permutation $(1,2)(3,4)$. The centralizer contains $(5,6,7)$,  $(1,3)(2,4)$ and $(1,2)(5,6)$; the involution  $(1,3)(2,4)$ commutes with the element of order three $(5,6,7)$, so  $(1,3)(2,4)$ is an $\bar h$--rotation. On the other hand $(1,3)(2,4)$ and  $(1,2)(5,6)$ do not commute and by \fullref{prop1}, this cannot occur.

Finally we consider the groups  ${\rm PSL}(2,q)$, with $q\geq 5$. The only central perfect extension of   ${\rm PSL}(2,q)$ with nontrivial center of order a power of two  is  ${\rm SL}(2,q)$, that contains a unique involution  that is central in the group and this is not possible. The only remaining possibility is that $Z(E)$ is trivial and $E$ is isomorphic to  ${\rm PSL}(2,q)$, with $q\geq 5$.

\medskip{\bf Step 6}\qua If $E$ has two components, $E$ is isomorphic to  $ {\rm SL}(2,q) \times_{\Z_2} {\rm SL}(2,q')$, with  $q$ and $q'$  odd prime powers greater than four.

 We consider $\tilde E=\tilde A \times \tilde B$ where $\tilde A$ and $\tilde B$ are two simple groups. By Step  3  $\tilde A$ and $\tilde B$ have sectional $2$--rank two. The simple groups with this property are:
$$\hbox{${\rm PSL} (2,q)$, for $q$ odd and $q\geq 5$; $M_{11}$;
$\Bbb A_7$; ${\rm PSL} (3,q)$ and  ${\rm PSU} (3,q)$, for $q$ odd.}$$ 
By Step 2 we recall that we have an involution $\tilde h$ in $\tilde E$ such that its centralizer is admissible.
We have that $\tilde h= (\tilde h_A, \tilde h_B)$ where $\tilde h_A\in \tilde A$ and $\tilde h_B\in \tilde B$. The centralizer of $C_{\smash{\tilde E}}(\tilde h)$ is the direct product of $\smash{C_{\tilde A}(\tilde h_A)}$ and  $\smash{C_{\tilde B}(\tilde h_B)}$. We remark that    $\tilde h_A$ and $\tilde h_B$ cannot be the identity of the group otherwise the  centralizer of $\tilde h$ is not solvable, so they are involutions.
The two centralizers $C_{\tilde A}(\tilde h_A)$ and  $ C_{\tilde B}(\tilde h_B)$ must be  admissible groups.
This condition excludes as components  $M_{11}$, ${\rm PSL} (3,q)$ and  ${\rm PSU} (3,q)$, with $q$ odd, because they do not contain any involution with admissible centralizer (see Step~5).

So we obtain that $\tilde A$ an $\tilde B$ are isomorphic to 
 ${\rm PSL} (2,q)$ or $\Bbb A_7$.
If $Z(E)$ is trivial, the centralizer of each involution in $E$ contains an elementary abelian group of order sixteen; moreover the group $\tilde E=E$  contains the involution $\bar h$ and the  centralizer of $\bar h$ cannot contain any elementary abelian group of order sixteen.
We can suppose that $Z(E)$ is not trivial, that is at least one between the components of $E$ is not simple. By Step 1, the center $Z(E)$ is a $2$--group. The  central perfect extensions of ${\rm PSL} (2,q)$ and $\A_7$ with center with order a power of two are ${\rm SL} (2,q)$ and $\A_7^*$ that contain a unique involution that is central in the groups. So $E$ cannot be  a direct product of its components otherwise the centralizer of each involution in $E$ contains a nonsolvable group.
We obtain  that $E=A\times_{{\Bbb Z}_2}B$ where $A,B\cong {\rm SL} (2,q)$ or $\A^*_7$. 

Finally we exclude $\smash{\A^*_7}$ as a possible component.
We consider  $\smash{\bar h=(\bar  h_A, \bar  h_B)}$, where $\smash{\bar h_A} \in A$ and $\smash{\bar h_B} \in  B$. The centralizer of $\bar h$ contains the centralizer of  $\bar h_A$ in $ A$ and the centralizer of  $\bar h_B$ in $ B$. If one between $\bar h_A$ and $\bar h_B$ is the identity or is an element of  order two, the centralizer of $\bar h$ is not solvable. To have an admissible centralizer for $\bar h$, we have to suppose that both $\bar h_A$ and $\bar h_B$ have order four (note that  $\bar h$ has order two). Any element of order four in $\A^*_7$ contains in its centralizer  noncommuting elements of order eight and three  which contradicts \fullref{prop1} and  \fullref{prop3} (see the \emph{Atlas of finite groups} \cite{A} for  the structure of  $ \A^*_7$).

\medskip{\bf Final step}\qua
 We denote by  $C=\smash{C_{\bar G}(E)}$  the centralizer of $E$ in $\bar G$.  Since $E$ is normal in $G$, the group $C$ is normal in $G$. Since $C$ is contained in the normalizer of $\bar h$ it  is isomorphic to a subgroup of  the semidirect product ${\Bbb Z}_2 \ltimes ({\Bbb Z}_{a}\times {\Bbb Z}_{b})$. The maximal subgroup of odd order in $C$ is unique and so is characteristic, thus it is normal in $G$. It follows that $C$ has to be a $2$--group otherwise $\O(G)$ is not maximal.
We denote by  $D$ the subgroup generated by $E$ and $C$; the subgroup $D$ is a central product $E\times_{Z(E)}C$.  The factor group of $\bar G$ by $D$ is a subgroup of ${\rm Out} (E)$, the outer automorphism group of $E$.

 Consider first the case $E\cong {\rm PSL}(2,q)$; in this case $D=E\times C$. In $E$ all the involutions are conjugate, so the centralizer in $\smash{\bar G}$ of each involution of $E$  is isomorphic to a subgroup of  the semidirect product ${\Bbb Z}_2 \ltimes ({\Bbb Z}_{a}\times {\Bbb Z}_{b})$.
The subgroup $E$  contains an elementary subgroup isomorphic to $ {\Bbb Z}_2 \times {\Bbb Z}_2$;  the subgroup $C$  centralizes  each involution in $E$. Since the  only possible abelian $2$--group with rank at least three, contained in  ${\Bbb Z}_2 \ltimes ({\Bbb Z}_{a}\times {\Bbb Z}_{b})$, is the elementary abelian subgroup of order eight, then either $C$ is trivial or $C\cong {\Bbb Z}_2$.  
The outer automorphism group of   ${\rm PSL}(2,q=p^n)$ is isomorphic to ${\Bbb Z}_2\times{\Bbb Z}_n$ \cite[p\,509]{S2}.

Suppose now that $E$ has two components; the factor group $D/Z(E)$ is isomorphic to $ E/Z(E)\times C/Z(E)$. Since $E/Z(E)$ has sectional $2$--rank four and $D/Z(E)$ has sectional $2$--rank at most four, it follows that $C/Z(E)$ has to be trivial and $E=D$. 
The set of the  components of $E$ is uniquely determined by the group and any automorphism of $E$ induces a permutation on the set of its  components \cite[Theorem 3.5, p\,7]{GLS2}; if   $E=A\times_{{\Bbb Z}_2}B$, then  the outer automorphism group of $E$ contains with index at most two a subgroup isomorphic to ${\rm Out}(A)\times  {\rm Out} (B)$ \cite[Lemma 3.23, p\,13]{GLS2}. The outer automorphism group of ${\rm SL}(2,q)$ is the same as that of  ${\rm PSL}(2,q)$ that  is isomorphic to ${\Bbb Z}_2\times{\Bbb Z}_n$.
This concludes the proof.

\section[Proof of \ref{thm3}]{Proof of \fullref{thm3}}\label{Section 5}

We denote by   $H$ (resp.\  $H'$)  the transformation group  of $K$ (resp.\ $K'$);  each nontrivial element of $H$  (resp.\  $H'$) fixes pointwise the same simple connected curve $\tilde K$ (resp.\ $\tilde K'$) in $M$ that is the preimage of $K$ (resp.\ $K'$) in $M$. Since $M$ is hyperbolic,  by Thurston's  orbifold geometrization theorem  \cite{BLPo}, we can suppose, up to conjugation,  that the  transformation groups are contained in $G$.

We note that $\tilde K$ and $\tilde K'$ do not coincide, even after conjugation. If   $n=m$ it  follows  from the fact that $K$ and $K'$ are inequivalent. If $n\neq m$ and $H'$ fixes pointwise $\tilde K$ we obtain some nontrivial symmetries of the knot $K$ which fix pointwise the knot and this is impossible by  the positive solution to the Smith Conjecture.

For each prime divisor $p$ of $n$ (resp.\ $m$) we denote by $H_p$ (resp.\ $H'_p$) the Sylow p-subgroup of $H$ (resp.\ $H'$). 
 
\medskip{\bf Step 1}\qua  Suppose that  a  subgroup  of $H$    with  order strictly greater than two normalizes a  subgroup of $H'$  with  order strictly greater than two,   then $H$ commutes elementwise with  $H'$; in particular   $K$ and $K'$ arise from the standard abelian construction. The same statement holds inverting the roles of $K$ and $K'$.

We denote by $B$ the subgroup of $H$  and by $B'$ the  subgroup of  $H'$. The subgroup  $B$ normalizes $B'$; since  $B$ has order  strictly greater than two, \fullref{prop1} implies that $B$ commutes elementwise with $B'$. The group $B'$ fixes setwise $\tilde K$; if $f\in B'$ we obtain that $fH f^{-1}$ fixes pointwise $\tilde K$. Since there exists at most one cyclic group of given order that fixes pointwise a connected curve,  we obtain that $B'$ normalizes $H$. By \fullref{prop1},  the groups  $B'$ and $H$  commute elementwise. Using the same argument  as before we obtain that  $H$ and $H'$   commute elementwise, this concludes the proof.  

\medskip{\bf Step 2}\qua Let $B$ be a subgroup of $G$ and let $p$ be an odd prime number such that $p$ divides the order of $B\cap H$ or the order of  $B\cap H'$. Then $S_p$, the Sylow $p$--subgroup of $B$,  is abelian of rank one or two; there are exactly one or two simple closed curves in $M$ that are fixed by some nontrivial element of $S_p$ with  connected fixed point set; the normalizer $N_G(S_p)$ of $S_p$ in $G$ is solvable.

\medskip{\bf Remark}\qua The statement of Step 2  may appear rather technical but it has the advantage that,  in this form,  it  applies directly throughout the remaining steps. 

\medskip Without loss of generality, we  suppose that $p$ divides the order of  $B\cap H$. Up to conjugation we can suppose that  $H_p\cap B$  is contained in $S_p$. 

We consider $N=N_{S_p}(H_p\cap B)$ the normalizer of $H_p\cap B$ in $S_p$. By \mbox{\fullref{prop1}} the group $N$ is abelian of rank at most two.
By Step 1 the group $N$ projects to a group of symmetries of $K$.  Since $M$ is hyperbolic, $K$ is a hyperbolic knot and in particular is not the unknotted circle. By the positive solution of the Smith conjecture,  $N/(H_p\cap B)$ is cyclic and there exists at most one connected simple closed curve fixed pointwise  by elements of $N/(H_p\cap B)$.
 An element of $N$, that is not contained in $H_p$ and  has  nonempty fixed point set,  projects to a nontrivial symmetry of $K$ with nonempty fixed point set; moreover  $H_p\cap B$ fixes setwise the fixed point set of any element of $N$. Thus in $N$ there exists at most one maximal cyclic subgroup different from $H_p\cap B$ with nonempty connected fixed point set.

If $f$ is  an element  of $S_p$ that normalizes  $N$, it acts by conjugation on the set of maximal cyclic subgroups with nonempty connected fixed point set. Since these groups are at most two and $p$ is odd, the action must be trivial and $f$ normalizes $ H_p\cap B$. We have that $N_{S_p}(N)=N$ and by  \cite[Theorem 1.6, p\,88]{S1} $S_p=N$. 

Finally we consider the normalizer $N_{G}(S_p)$ of $S_p$ in $G$. The group $N_{G}(S_p)$ acts by conjugation on the set of maximal cyclic subgroups of $S_p$ with nonempty connected fixed point set. Since these groups are at most two, the normalizer  $N_{G}(S_p)$ contains with index at most two   $N_{G}(H_p)$ that is solvable.
This concludes the proof. 

\medskip{\bf Step 3}\qua Let $B$ (resp.\ $B'$) be a subgroup of $H$ (resp.\ $H'$) such that the order of $B$ (resp.\ $B'$) is not a power of two. If $B$ and $B'$   generate a subgroup $B\cdot B'$ of $G$ that does not contain any involution with connected and nonempty fixed point set, then $K$ and $K'$  arise from the standard abelian construction. In particular if $G$ does not contain any involution with connected and nonempty fixed point set, then $K$ and $K'$  arise from the standard abelian construction.

Let $p$ be an odd  prime number that divides the order of $B$, the subgroup $H\cap B$ contains a nontrivial $p$--group. We denote by $S_p$  a   $p$--Sylow of $B\cdot B'$, $S_p$  is abelian of rank at most two.

If $p$ divides also the order of $B'$, we can suppose that a nontrivial subgroup of $B$ and a nontrivial subgroup of $B'$ are  contained in the same Sylow $p$--subgroup of $B\cdot B'$. By Step 1 this implies that $K$ and $K'$ arise from the standard abelian construction.

We can suppose that an odd prime number $q$, different from $p$, divides the order of $B'$.  By   Step 2, we deduce that $S_p$ contains exactly one or two maximal cyclic subgroups with nonempty connected fixed point set; up to conjugation we can suppose that one of these groups is $S_p\cap H$. 
We consider $N$ the normalizer of $S_p$ in   $B\cdot B'$. The group $N$ acts by conjugation on the set of the maximal cyclic subgroups with nonempty connected fixed point set; $N$ contains   with index at most two $N_0$, the normalizer of $S_p\cap H$ in $B\cdot B'$. We recall that $H$ fixes pointwise $\tilde K$ that is a simple closed curve.
 
We prove that $N_0$ is abelian. Suppose that $N_0$  contains $t$,  an involution with nonempty fixed point set that acts as a reflection on $\tilde K$.
Since $t$ fixes setwise $\tilde K$, it normalizes $H$ and projects to a strong inversion of the knot $K$. Any  strong inversion of $K$ has connected fixed point set; since $K$ is connected, also  $t$ has connected fixed point set.  We suppose that  $B\cdot B'$ does not contain any involution with nonempty connected fixed point set, so each element of  the group  $N_0$ acts as a  rotation on $\tilde K$ and it is abelian.

Now we prove that $N=N_0$. If $N\neq N_0$, there exists an element $f\in N$ such  that $f (S_p\cap H) f^{-1}\neq (S_p \cap H)$. The group $f (S_p\cap H)f^{-1}$ and  $S_p \cap H$ commute elementwise. The fixed point set of $f (S_p\cap H) f^{-1}$ is $f(\smash{\tilde K})$, a simple closed curve that is distinct from $\smash{\tilde K}$. We consider the group $\smash{fH  f^{-1}}$ that fixes pointwise the simple closed curve $f(\tilde K)$; since $f(\tilde K)$ is distinct from $\tilde K$ the groups $fH  f^{-1}$ and $H$ intersect trivially. Moreover by  Step 1, the groups $fH  f^{-1}$ and $H$ commute elementwise. A generator of   $fH  f^{-1}$ projects to a cyclic symmetry of  $K$ with  order $n$ and  with nonempty connected fixed point set. Finally we obtain also that the associated quotient link is symmetric; in fact $f$ normalizes the group generated by  $fH  f^{-1}$ and $H$ and it projects to the quotient link exchanging the two components. If $N\neq N_0$, the knot $K$ should be self-symmetric and this is excluded by hypothesis.

We have obtained that the normalizer of $S_p$ in $B\cdot B'$ is abelian, in particular $S_p$ is contained in the center of its normalizer. By \cite[Theorem 2.10, p\,143]{S2} $B\cdot B'$ splits as a semidirect product $U\rtimes  S_p$. We have supposed that there exists $q$ different from $p$ such that $q$ divides the order of $B'$. Any  Sylow $q$--subgroup is contained in $U$ and $S_p$ acts by conjugation on the set of Sylow $q$--subgroups. Since $p$ does not divide the order of $U$, it follows that some orbit has only one element. We obtain a Sylow $q$--subgroup $S_q$ that is normalized by $S_p$; up to conjugation we can suppose that the intersection of $S_q$ and $B'$ is not trivial. By Step 2  we obtain that $S_p\cap B$ normalizes $S_q\cap B'$; by Step 1 we obtain that $K$ and $K'$ arise from the standard abelian construction.  
 
\medskip{\bf Step 4}\qua    Let $B$ (resp.\ $B'$) be a subgroup of $H$ (resp.\ $H'$) such that the order of $B$ (resp.\ $B'$) is not a power of two. If $B$ and $B'$   generate a solvable subgroup $B\cdot B'$ of $G$, then $K$ and $K'$  arise from the standard abelian construction.

As in step 3 we can suppose that there exist two different odd primes $p$ and $q$, such that $p$ divides the order of $B$ and $q$ divides the order of $B'$.
By Sylow theorems for solvable groups, we obtain  that there exists $A$,  a subgroup of $B\cdot B'$ with order $p^{\alpha}q^{\beta}$, that contains a Sylow $p$--subgroup and a Sylow $q$--subgroup of   $B\cdot B'$. Up to conjugation we can suppose that the intersection of $A$ both with $B$ and with $B'$ is not trivial. The group $A$ does not contain any involution, so applying Step 3 to $A$ we obtain that $K$ and $K'$ arise from the standard abelian construction.    

\medskip{\bf Step 5}\qua If $E$ is not trivial, either $K$ and $K'$ arise from the standard abelian construction  or there exists in $G$ an involution $h$ with nonempty connected fixed point set such that $h\O(G)\in E$.

If $G$ does not contain any involution with nonempty connected fixed point set,   by Step~3 the knots $K$ and $K'$ arise from the standard abelian construction.  

We can suppose that $G$ contains one involution with nonempty connected fixed point set. We denote by ${\cal E}$ the preimage of $E$ in $G$ with respect to the projection of $G$ onto $O/\O(G)$; by \fullref{thm1} the factor group $G/{\cal E}$ is solvable. 

Suppose that the group $\cal E$ does not contain any involution with nonempty connected fixed point set. 
Let $p$ be an odd prime number that divides $n$ and let $q$ be an odd prime number that divides $m$. By the Sylow Theorem for solvable groups, there exists a subgroup ${\cal E}'$ of $G$ that contains ${\cal E}$, with the following properties:
\begin{itemize}
\item  ${\cal E}'$contains a Sylow $p$--subgroup and a Sylow $q$--subgroup of $G$; 

\item the factor group ${\cal E}'/{\cal E}$ has order $p^{\alpha}q^{\beta}$. 
\end{itemize}
All the involutions in  ${\cal E}'$ are contained in ${\cal E}$, so ${\cal E}'$ does not contain any involution with nonempty connected fixed point set. Moreover, up to conjugation, we can suppose that ${\cal E}'$ contains a covering transformation of order $p$ (resp.\ $q$) of $K$ (resp.\ $K'$).
By Step 3 the knots $K$ and $K'$ arise from the standard abelian construction.

\medskip{\bf Step 6}\qua If $E$ is trivial, either $K$ and $K'$ arise from the standard abelian construction  or there exists a normal subgroup $N$ of $G$ such that $N$ is solvable and $G/N$ is isomorphic to a subgroup of  ${\rm GL}(4,2)$. In this case if $K$ and $K'$ do not arise   from the standard abelian construction  all  prime divisors of  $n$ and  $m$ are  contained in $\{2,3,5,7\}.$

If $G$ does not contain any involution with connected fixed point set, by Step 3, the knots $K$ and $K'$ arise from the standard abelian construction.

If $G$ contains an involution with connected fixed point set, by \fullref{thm1} there exists a normal subgroup $N$ of $G$ such that $N$ is solvable and $G/N$ is isomorphic to a subgroup of  ${\rm GL}(4,2)$.

Finally we prove that, if the intersection group $H\cap N$ contains a nontrivial element of odd order, then $K$ and $K'$ arise from the standard abelian construction. Let $p$ be an odd  prime number that divides the order of    $H\cap N$, let  $q$ be an odd  prime number that divides $m$. We can suppose that $p$ is different from $q$, otherwise by Step 2 the knots $K$ and $K'$ arise from the standard abelian construction.

 By the Sylow Theorem applied to $G/N$, there exists a subgroup $ N'$ of $G$ that contains $ N$, with the following properties:
\begin{itemize}
\item $N'$ is solvable; 
\item $N'$  contains  a Sylow $q$--subgroup of $G$.
\end{itemize}
By Step 4 the knots $K$ and $K'$ arise from the standard abelian construction.

The same property holds if  $H'\cap N$ contains a nontrivial element of odd order.  

Since the order of ${\rm GL}(4,2)$ is $7 \cdot 5 \cdot 3^2 \cdot 2 ^6$, the Sylow $p$--subgroups of $G$ are contained in $N$ when $p\neq 2,\,3,\,5,\,7$. 
So if $K$ and $K'$ do not arise from the standard abelian construction, all  prime divisors of  $n$ and  $m$ are  contained in $\{2,3,5,7\}.$
 
\subsubsection*{Acknowledgements}
I am  deeply indebted to  Marco Reni. He   was the first that considered the possibility  to extend   the result obtained by him and Zimmermann \cite{RZ1} for simple groups to  the nonsolvable case but  he  sadly died in June 2000 and he could not continue his work. Bruno Zimmermann kindly communicated to me some of his ideas  and these suggestions are the starting point to write  this paper. The author is also grateful to the referee for many helpful comments that improved the article.

\bibliographystyle{gtart}
\bibliography{link}

\end{document}